\documentclass{amsart}

\usepackage{enumerate, amsmath, amsfonts, amssymb, amsthm, wasysym, graphics, graphicx, xcolor, colortbl, url, hyperref, hypcap, ifthen, xargs, pdflscape, caption, tikz, rotate, multicol, multirow}
\hypersetup{colorlinks=true, citecolor=darkblue, linkcolor=darkblue}

\newtheorem{theorem}{Theorem}[section]
\newtheorem{proposition}[theorem]{Proposition}
\newtheorem{corollary}[theorem]{Corollary}
\newtheorem{lemma}[theorem]{Lemma}

\theoremstyle{definition}
\newtheorem{definition}[theorem]{Definition}
\newtheorem{example}[theorem]{Example}

\newtheorem{remark}[theorem]{Remark}

\newcommand{\set}[2]{\left\{ #1 \;\middle|\; #2 \right\}} % set notation
\newcommand{\ssm}{\smallsetminus} % small set minus
\newcommand{\symdif}{\triangle} % symmetric difference
\newcommand{\eqdef}{\mbox{\,\raisebox{0.2ex}{\scriptsize\ensuremath{\mathrm:}}\ensuremath{=}\,}} % :=
\newcommand{\R}{\mathbb{R}} % reals
\newcommand{\Z}{\mathbb{Z}} % integers
\newcommand{\cA}{\mathcal A} % calligraphic A
\newcommand{\fS}{\mathfrak S} % symmetric group
\newcommand{\sqc}{{\mathrm c}} % square c
\newcommand{\Q}{{\mathrm Q}} % square Q
\newcommand{\sqP}{{\mathrm P}} % square P
\newcommand{\tQ}{\widetilde{\mathrm Q}} % tilde square Q
\newcommand{\tsqq}{\widetilde{\mathrm q}} % tilde square q
\newcommand{\tI}{\widetilde{I}} % tilde I
\newcommand{\ti}{\tilde{i}} % tilde i
\newcommand{\tj}{\tilde{j}} % tilde j
\newcommand{\Qc}{\Q_\sqc} % square Q

\newcommandx{\cdeg}[3][1=\alpha, 2=\beta, 3={}]{(#1 \, \|_{#3} \, #2)} % compatibility degree
\newcommandx{\vcoeff}[3][1=\alpha, 2=\beta, 3={}]{\{#1 \, \|_{#3} \, #2\}} % compatibility degree on positions
\newcommandx{\cross}[2][1=\theta, 2=\delta]{[\,#1 \, \| \, #2 \,]} % crossing vector
\newcommand{\dvector}{\mathbf{d}} % denominator vector
\newcommand{\rootsToVar}{\phi} % bijection from roots to variables by FZ
\newcommand{\rootsToVarc}{\rootsToVar_c} % bijection from roots to variables
\newcommand{\posToVar}{\psi} % bijection from positions to variables
\newcommand{\posToVarc}{\posToVar_\sqc} % bijection from positions to variables
\newcommand{\posToRoots}{\vartheta} % bijection from positions to roots
\newcommand{\posToRootsc}{\posToRoots_\sqc} % bijection from positions to roots
\newcommand{\posToDiag}{\zeta} % bijection from positions to diagonals
\newcommandx{\posToDiagc}[1][1=]{\posToDiag_{\sqc}^{#1}} % bijection from positions to diagonals
\newcommand{\orbToVar}{\Psi} % bijection from orbits of positions to variables
\newcommand{\diagToVar}{\chi} % bijection from diagonals to variables
\newcommand{\tausqc}{\tau_\sqc} % rotation
\newcommand{\tauc}{\tau_c} % rotation
\newcommand{\weightedQuiverc}{\mathcal{Q}_c} % weighted quiver

\newcommand{\subwordComplex}[1][\Q,\pi]{\mathcal{SC}(#1)} % subword complex
\newcommand{\wordprod}[2]{\sigma^{#1}_{#2}} % product of the letters in Q \ssm I
\newcommand{\Root}[2]{\mathsf{r}(#1,#2)} % the root of #1 at position #2
\newcommand{\Roots}[1]{\mathsf{R}(#1)} % the root configuration of #1
\newcommand{\wo}{w_\circ} % the longest element w_o
\newcommand{\woc}{{\mathrm w}_\circ(\sqc)} % w_o(c)
\newcommand{\cwoc}{{\sqc\woc}} % c.w_o(c)
\newcommand{\translation}{\tau} % translation map
\newcommand{\diagD}[2]{\ifthenelse{\equal{#2}{0}}{#1^{\textsc l}}{#1^{\textsc r}}} % diagonal
\newcommand{\bigfrac}[2]{$\displaystyle{\frac{#1}{#2}}$} % big fraction in a tabular environnement
\definecolor{Gray}{gray}{0.9} % light gray
\newcolumntype{g}{>{\columncolor{Gray}}c}
\newcolumntype{x}{@{\hspace{.1cm}}c@{\hspace{.1cm}}}

\newcommand{\ie}{\textit{i.e.}~} % id est
\newcommand{\eg}{\textit{e.g.}~} % exempli gratia
\newcommand{\fref}[1]{Figure~\ref{#1}} % reference figures
\graphicspath{{figures/}}

\definecolor{darkblue}{rgb}{0,0,0.7} % darkblue color
\newcommand{\darkblue}{\color{darkblue}} % darkblue command
\newcommand{\defn}[1]{\emph{\darkblue #1}} % emphasis of a definition

\usepackage{todonotes}

\numberwithin{equation}{section}

\title[Denominator vectors in cluster algebras of finite type]{Denominator vectors and compatibility degrees in cluster algebras of finite type}

\author[C.~Ceballos]{Cesar Ceballos$^{\star}$} 
\address[C.~Ceballos]{Department of Mathematics and Statistics, York University, Toronto, Ontario M3J 1P3, CANADA}
\email{ceballos@mathstat.yorku.ca}
\urladdr{http://garsia.math.yorku.ca/~ceballos/}
\thanks{$^\star$CC was supported by DFG via the Research Training Group ``Methods for Discrete Structures" and the Berlin Mathematical School.}

\author[V.~Pilaud]{Vincent Pilaud$^{\ddagger}$} 
\address[V.~Pilaud]{CNRS \& LIX, \'Ecole Polytechnique, Palaiseau}
\email{vincent.pilaud@lix.polytechnique.fr}
\urladdr{http://www.lix.polytechnique.fr/~pilaud/}
\thanks{$^\ddagger$VP was supported by the spanish MICINN grant MTM2011-22792, by the French ANR grant EGOS 12 JS02 002 01, and by the European Research Project ExploreMaps~(ERC~StG~208471).}


\begin{document}

\maketitle

\begin{abstract}
We present two simple descriptions of the denominator vectors of the cluster variables of a cluster algebra of finite type, with respect to any initial cluster seed: one in terms of the compatibility degrees between almost positive roots defined by S.~Fomin and A.~Zelevinsky, and the other in terms of the root function of a certain subword complex. These descriptions only rely on linear algebra. They provide two simple proofs of the known fact that the $d$-vector of any non-initial cluster variable with respect to any initial cluster  seed has non-negative entries and is different from zero. 
\end{abstract}

%%%%%%%%%%%%%%%%%%%%%%%%%%%%%%%%%%%%%%

\section{Introduction}
\label{sec:introduction}

\defn{Cluster algebras} were introduced by S.~Fomin and A.~Zelevinsky in the series of papers~\cite{FominZelevinsky-ClusterAlgebrasI, FominZelevinsky-ClusterAlgebrasII, FominZelevinsky-ClusterAlgebrasIII, FominZelevinsky-ClusterAlgebrasIV}. They are commutative rings generated by a (possibly infinite) set of \defn{cluster variables}, which are grouped into overlapping \defn{clusters}. The clusters can be obtained from any \defn{initial cluster seed}~$X = \{x_1, \dots, x_n\}$ by a mutation process. Each mutation exchanges a single variable~$y$ to a new variable~$y'$ satisfying a relation of the form~${yy' = M_+ + M_-}$, where~$M_+$ and~$M_-$ are monomials in the variables involved in the current cluster and distinct from~$y$ and~$y'$. The precise content of these monomials~$M_+$ and~$M_-$ is controlled by a combinatorial object (a skew-symmetrizable matrix, or equivalently a weighted quiver~\cite{Keller}) which is attached to each cluster and is also transformed during the mutation. We refer to~\cite{FominZelevinsky-ClusterAlgebrasI} for the precise definition of these joint dynamics. In~\cite[Theorem~3.1]{FominZelevinsky-ClusterAlgebrasI}, S.~Fomin and A.~Zelevinsky proved that given any initial cluster seed $X = \{x_1, \dots, x_n\}$, the cluster variables obtained during this mutation process are \defn{Laurent polynomials} in the variables~$x_1, \dots, x_n$. That is to say, every non-initial cluster variable~$y$ can be written in the form
$$ y = \frac{F(x_1,\dots,x_n)}{x_1^{d_1} \cdots x_n^{d_n}} $$
where~$F(x_1,\dots,x_n)$ is a polynomial which is not divisible by any variable~$x_i$ for~${i \in [n]}$. This intriguing property is called \defn{Laurent Phenomenon} in cluster algebras~\cite{FominZelevinsky-ClusterAlgebrasI}. The \defn{denominator vector} (or \defn{$d$-vector} for short) of the cluster variable~$y$ with respect to the initial cluster seed~$X$ is the vector~$\dvector(X,y) \eqdef (d_1,\dots,d_n)$. The $d$-vector of the initial cluster variable $x_i$ is~$\dvector(X,x_i) \eqdef -e_i \eqdef (0,\dots,-1,\dots,0)$ by definition.

Note that we think of the cluster variables as a set of variables satisfying some algebraic relations. These variables can be expressed in terms of the variables in any initial cluster seed~$X = \{x_1,\dots,x_n\}$ of the cluster algebra. Starting from a different cluster seed~$X' = \{x'_1, \dots, x'_n\}$ would give rise to an isomorphic cluster algebra, expressed in terms of the variables~$x'_1, \dots, x'_n$ of this seed. Therefore, the $d$-vectors of the cluster variables depend on the choice of the initial cluster seed~$X$ in which the Laurent polynomials are expressed. This dependence is explicit in the notation~$\dvector(X,y)$. Note also that since the denominator vectors do not depend on coefficients, we restrict our attention to coefficient-free cluster algebras.

In this paper, we only consider finite type cluster algebras, \ie cluster algebras whose mutation graph is finite. They were classified in~\cite[Theorem~1.4]{FominZelevinsky-ClusterAlgebrasII} using the Cartan-Killing classification for finite crystallographic root systems. In~\cite[Theorem~1.9]{FominZelevinsky-ClusterAlgebrasII}, S.~Fomin and A.~Zelevinsky proved that in the cluster algebra of any given finite type, with a bipartite quiver as initial cluster seed,
\begin{enumerate}[(i)]
\item there is a bijection~$\rootsToVar$ from almost positive roots to cluster variables, which sends the negative simple roots to the initial cluster variables;
\item the $d$-vector of the cluster variable~$\rootsToVar(\beta)$ corresponding to an almost positive root~$\beta$ is given by the vector~$(b_1, \dots, b_n)$ of coefficients of the root~$\beta = \sum b_i\alpha_i$ on the linear basis~$\Delta$ formed by the simple roots~$\alpha_1,\dots,\alpha_n$; and
\item these coefficients coincide with the \defn{compatibility degrees}~$\cdeg[\alpha_i][\beta]$ defined in \cite[Section~3.1]{FominZelevinsky-YSystems}.
\end{enumerate}

These results were extended to all cluster seeds corresponding to Coxeter elements of the Coxeter group (see \eg~\cite[Theorem~3.1 and Section~3.3]{Keller}). More precisely, assume that the initial seed is the cluster~$X_c$ corresponding to a Coxeter element~$c$ (its associated quiver is the Coxeter graph oriented according to~$c$). Then one can define a bijection~$\rootsToVarc$ from almost positive roots to cluster variables such that the $d$-vector of the cluster variable~$\rootsToVarc(\beta)$ corresponding to~$\beta$, with respect to the initial cluster seed~$X_c$, is still given by the vector~$(b_1, \dots, b_n)$ of coordinates of~$\beta = \sum b_i\alpha_i$ in the basis~$\Delta$ of simple roots. Under this bijection, the collections of almost positive roots corresponding to clusters are called \defn{$c$-clusters} and were studied by N.~Reading~\cite[Section~7]{Reading-CoxeterSortable}.

In this paper, we provide similar interpretations for the denominators of the cluster variables of any finite type cluster algebra with respect to \defn{any initial cluster seed} (acyclic or not):
\begin{enumerate}[(i)]
\item Our first description (Corollary~\ref{thm:d_vectors-compatibilty}) uses compatibility degrees: if~$\{\beta_1,\dots,\beta_n\}$ is the set of almost positive roots corresponding to the cluster variables in any initial cluster seed~$X = \{\rootsToVar(\beta_1),\dots,\rootsToVar(\beta_n)\}$, then the $d$-vector of the cluster variable~$\rootsToVar(\beta)$ corresponding to an almost positive root~$\beta$, with respect to the initial cluster seed~$X$, is still given by the vector of compatibility degrees~$(\cdeg[\beta_1][\beta], \dots, \cdeg[\beta_n][\beta])$ of~\cite[Section~3.1]{FominZelevinsky-YSystems}. We also provide a refinement of this result parametrized by a Coxeter element~$c$, using the bijection~$\rootsToVarc$ together with the notion of $c$-compatibility degrees (Corollary~\ref{thm:d_vectors-c_compatibilty}).
\item Our second description (Corollary~\ref{thm:d-vector}) uses the recent connection~\cite{CeballosLabbeStump} between the theory of cluster algebras of finite type and the theory of subword complexes, initiated by A.~Knutson and E.~Miller~\cite{KnutsonMiller-subwordComplex}. We describe the entries of the $d$-vector in terms of certain coefficients given by the root function of a subword complex associated to a certain word.
\end{enumerate}

Using these results, we provide two alternative proofs of the known fact that, in a cluster algebra of finite type, the $d$-vector of any non-initial cluster variable with respect to any initial cluster seed is non-negative and not equal to zero (Corollary~\ref{coro:positive}).

Even if we restrict here to crystallographic finite types since we deal with cluster variables of the associated cluster algebras, all the results not involving cluster variables remain valid for any arbitrary finite type. This includes in particular the results about almost positive roots, $c$-clusters, $c$-compatibility degrees, rotation maps, and their counterparts in subword complexes.
We also highlight that subword complexes played a fundamental role in the results of this paper. Even if the main result describing denominator vectors in terms of compatibility degrees can be proved independently, we would not have been able to find it without using the subword complex approach. 

Finally, we also provide explicit geometric interpretations of denominator vectors for the classical types~$A$,~$B$,~$C$ and~$D$ in Section~\ref{sec:geom-interpretations}. Our description of type~$D$ cluster algebras is new and will be explored further in a forthcoming paper.
%%%%%%%%%%%%%%%%%%%%%%%%%%%%%%%%%%%%%%

\section{Preliminaries}
\label{sec:preliminaries}

Let~$(W,S)$ be a finite crystallographic Coxeter system of rank~$n$. We consider a root system~$\Phi$, with simple roots~$\Delta \eqdef \{\alpha_1,\dots,\alpha_n\}$, positive roots~$\Phi^+$, and almost positive roots~$\Phi_{\ge-1} \eqdef \Phi^+ \cup -\Delta$. We refer to~\cite{Humphreys} for a reference on Coxeter groups and root systems.

Let~$\cA(W)$ denote the cluster algebra associated to type~$W$, as defined in~\cite{FominZelevinsky-ClusterAlgebrasII}. Each cluster is formed by~$n$ cluster variables, and is endowed with a weighted quiver (an oriented and weighted graph on~$S$) which controls the cluster dynamics. Since we will not make extensive  use of it, we believe that it is unnecessary to recall here the precise definition of the quiver and cluster dynamics, and we refer to~\cite{FominZelevinsky-ClusterAlgebrasI, Keller} for details.

Let~$c$ be a Coxeter element of~$W$, and~$\sqc \eqdef (c_1, \cdots, c_n)$ be a reduced expression of~$c$. The element~$c$ defines a particular weighted quiver~$\weightedQuiverc$: the Coxeter graph of the Coxeter system~$(W,S)$ directed according to the order of appearance of the simple reflections in~$c$. We denote by~$X_c$ the cluster seed whose associated quiver is~$\weightedQuiverc$. Let $\woc \eqdef (w_1, \cdots, w_N)$ denote the \defn{$\sqc$-sorting word} for~$\wo$, \ie the lexicographically first subword of the infinite word~$\sqc^{\infty}$ which represents a reduced expression for the longest element $\wo \in W$. We consider the word~$\Qc \eqdef \cwoc$ and denote by ${m \eqdef n+N}$ the length of this word.

%%%%%%%%%%%

\subsection{Cluster variables, almost positive roots, and positions in the word~$\Qc$}
\label{subsec:bijections}

We recall here the above-mentioned bijections between cluster variables, almost positive roots and positions in the word~$\Qc$. We will see in the next sections that both the clusters and the $d$-vectors (expressed on any initial cluster seed~$X$) can also be read off in these different contexts. \fref{fig:bijections} summarizes these different notions and the corresponding notations. We insist that the choice of the Coxeter element~$c$ and the choice of the initial cluster~$X$ are not related. The former provides a labeling of the cluster variables by the almost positive roots or by the positions in~$\Qc$, while the latter gives an algebraic basis to express the cluster variables and to assign them $d$-vectors.

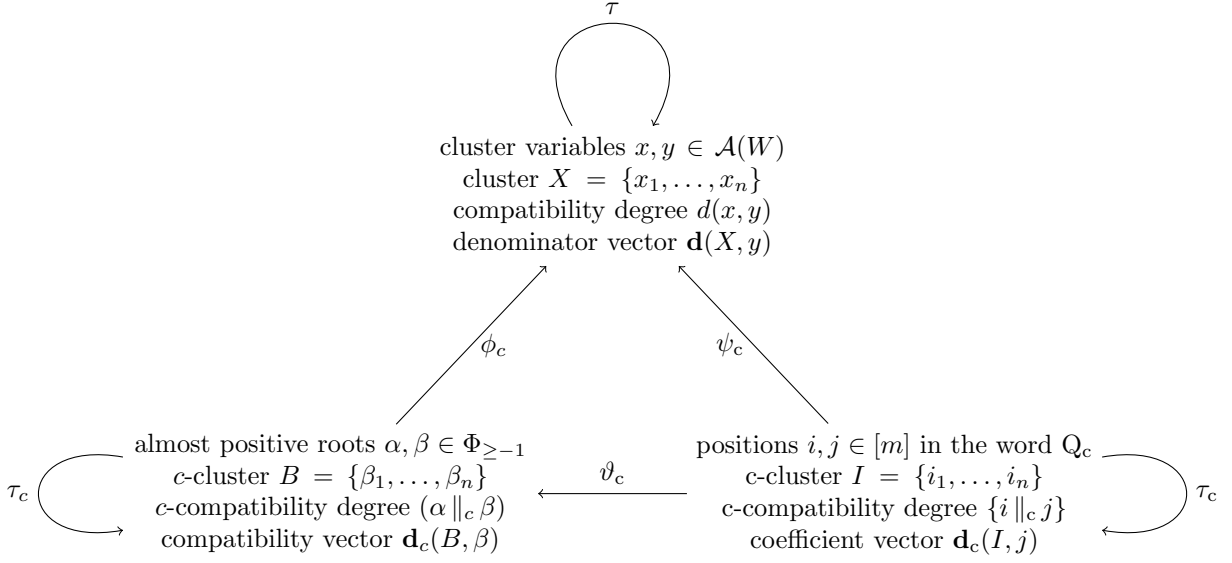
\begin{figure}
	\hspace*{-1.9cm}
	\begin{tikzpicture}[block_center/.style ={text width=15em, text centered}]
		\matrix [row sep=3em, column sep=-5em]
			{ & \node [block_center] (cluster) {cluster variables $x, y \in \cA(W)$ \\ cluster $X =\{x_1,\dots,x_n\}$ \\ compatibility degree $d(x,y)$ \\ denominator vector $\dvector(X,y)$}; \\ \\
			  \node [block_center] (roots) {almost positive roots~$\alpha, \beta \in \Phi_{\ge -1}$ \\ $c$-cluster $B = \{\beta_1, \dots, \beta_n\}$ \\ $c$-compatibility degree $\cdeg[\alpha][\beta][c]$ \\ compatibility vector $\dvector_c(B,\beta)$}; & & \node [block_center] (subword) {positions $i,j \in [m]$ in the word~$\Qc$ \\ $\sqc$-cluster $I = \{i_1, \dots, i_n\}$ \\ $\sqc$-compatibility degree $\vcoeff[i][j][\sqc]$ \\ coefficient vector $\dvector_{\sqc}(I,j)$}; \\
			};
		\path[->, font=\normalsize] (roots) edge node [right] {$\rootsToVarc$} (cluster);
		\path[->, font=\normalsize] (subword) edge node [left] {$\posToVarc$} (cluster);
		\path[->, font=\normalsize] (subword) edge node [above] {$\posToRootsc$} (roots);
		\path[->, font=\normalsize] (cluster) edge [loop above, in=60, out=120, looseness=5] node [above] {$\tau$} (cluster);
		\path[->, font=\normalsize] (roots) edge [loop below, in=190, out=170, looseness=4] node [left] {$\tauc$} (roots);
		\path[->, font=\normalsize] (subword) edge [loop below, in=350, out=10, looseness=4] node [right] {$\tausqc$} (subword);
	\end{tikzpicture}
	\caption{Three different contexts for cluster algebras of finite type, their different notions of compatibility degrees, and the bijections between them. See Sections~\ref{subsec:bijections}, \ref{subsec:rotation} and~\ref{subsec:compatibilityDegrees} for definitions.}
	\label{fig:bijections}
\end{figure}

First, there is a natural bijection between cluster variables and almost positive roots, which can be parametrized by the Coxeter element~$c$. Start from the initial cluster seed~$X_c$ associated to the weighted quiver~$\weightedQuiverc$ corresponding to the Coxeter element~$c$. Then the $d$-vectors of the cluster variables of~$\cA(W)$ with respect to the initial seed~$X_c$ are given by the almost positive roots~$\Phi_{\ge-1}$. This defines a bijection
$$\rootsToVarc : \Phi_{\ge-1} \longrightarrow \{\text{cluster variables of } \cA(W)\}$$
from almost positive roots to cluster variables. Notice that this bijection depends on the choice of the Coxeter element~$c$. When~$c$ is a bipartite Coxeter element, it is the bijection~$\rootsToVar$ of S.~Fomin and A.~Zelevinsky~\cite[Theorem~1.9]{FominZelevinsky-ClusterAlgebrasII} mentioned above. Transporting the structure of the cluster algebra~$\cA(W)$ through the bijection~$\rootsToVarc$, we say that a subset~$B$ of almost positive roots forms a \defn{$c$-cluster} iff the corresponding subset of cluster variables~$\rootsToVarc(B)$ forms a cluster of~$\cA(W)$. The collection of $c$-clusters forms a simplicial complex on the set~$\Phi_{\ge-1}$ of almost positive roots called the \defn{$c$-cluster complex}. This complex was described in purely combinatorial terms by N.~Reading in~\cite[Section~7]{Reading-CoxeterSortable}. Given an initial $c$-cluster seed~$B \eqdef \{\beta_1,\dots,\beta_n\}$ in~$\Phi_{\ge-1}$ and an almost positive root~$\beta$, we define the $d$-vector of~$\beta$ with respect to~$B$ as
$$\dvector_c(B,\beta) \eqdef \dvector \big( \rootsToVarc(B), \rootsToVarc(\beta) \big).$$
If~$c$ is a bipartite Coxeter element, then we speak about classical clusters and omit~$c$ in the previous notation to write~$\dvector(B,\beta)$.

Second, there is a bijection
$$\posToRootsc : [m] \longrightarrow \Phi_{\ge-1}$$
from the positions in the word $\Qc = \cwoc$ to the almost positive roots as follows. The letter~$c_i$ of~$\sqc$ is sent to the negative root~$-\alpha_{c_i}$, while the letter~$w_i$ of~$\woc$ is sent to the positive root~$w_{1} \cdots w_{i-1}(\alpha_{w_i})$. To be precise, note that this bijection depends not only on the Coxeter element~$c$, but also on its reduced expression~$\sqc$. This bijection was defined by C.~Ceballos, \mbox{J.-P.~Labb\'e} and C.~Stump in~\cite[Theorem~2.2]{CeballosLabbeStump}.

Composing the two maps described above provides a bijection
$$\posToVarc : [m] \longrightarrow \{\text{cluster variables of } \cA(W)\}$$
from positions in the word~$\Qc$ to cluster variables (precisely defined by $\posToVarc \eqdef \rootsToVarc \circ \posToRootsc$). Transporting the structure of~$\cA(W)$ through the bijection~$\posToVarc$, we say that a subset~$I$ of positions in~$\Qc$ forms a \defn{$\sqc$-cluster} iff the corresponding cluster variables~$\posToVarc(I)$ form a cluster of~$\cA(W)$. Moreover, given an initial $\sqc$-cluster seed~$I \subseteq [m]$ in~$\Qc$ and a position~$j \in [m]$ in $\Qc$, we define the $d$-vector of~$j$ with respect to~$I$ as
$$\dvector_\sqc(I,j) \eqdef \dvector \big( \posToVarc(I), \posToVarc(j) \big).$$
It turns out that the $\sqc$-clusters can be read off directly in the word~$\Qc$ as follows.

\begin{theorem}[\protect{\cite[Theorem~2.2 and Corollary~2.3]{CeballosLabbeStump}}]
\label{thm:CLS_cluster_complexes}
A subset~$I$ of positions in~$\Qc$ forms a $\sqc$-cluster in~$\Qc$ if and only if the subword of~$\Qc$ formed by the complement of~$I$ is a reduced expression for~$\wo$.
\end{theorem}

\begin{remark}
The previous theorem relates $c$-cluster complexes to subword complexes as defined by A.~Knutson and E.~Miller~\cite{KnutsonMiller-subwordComplex}. Given a word~$\Q$ on the generators~$S$ of~$W$ and an element~$\pi \in W$, the \defn{subword complex}~$\subwordComplex$ is the simplicial complex whose faces are subwords~$\sqP$ of~$\Q$ such that the complement~$\Q \ssm \sqP$ contains a reduced expression of~$\pi$. See~\cite{CeballosLabbeStump} for more details on this connection.
\end{remark}

%%%%%%%%%%%

\subsection{The rotation map}
\label{subsec:rotation}

In this section we introduce a rotation map $\tausqc$ on the positions in the word~$\Qc$, and naturally extend it to a map on almost positive roots and cluster variables using the bijections of Section~\ref{subsec:bijections} (see \fref{fig:bijections}). The rotation map plays the same role for arbitrary finite type as the rotation of the polygons associated to the classical types $A$, $B$, $C$ and~$D$, see~e.g.~\cite[Theorem~8.10]{CeballosLabbeStump}.

\begin{definition}[Rotation maps]
\label{def:rotation_map}
The rotation
$$\tausqc: [m] \longrightarrow [m]$$
is the map on the positions in the word $\Qc$ defined as follows. If $q_i=s$, then $\tausqc(i)$ is defined as the position in $\Qc$ of the next occurrence of $s$ if possible, and as the first occurrence of $\wo s \wo$ otherwise.

Using the bijection $\posToRootsc$ from the positions in the word $\Qc$ to almost positive roots, this rotation can also be regarded as a map from almost positive roots to almost positive roots. For simplicity, we abuse of notation and also write
$$\tauc: \Phi_{\geq-1} \longrightarrow \Phi_{\geq-1}$$
for the composition $\posToRootsc  \circ  \tausqc  \circ  \posToRootsc^{-1}$. This composition can be expressed purely in terms of roots as 
$$ 
\tauc (\alpha) =
\begin{cases}
	c_1\cdots c_{i-1} ( \alpha_{c_i}) &\text{ if } \alpha = -\alpha_{c_i}, \\
	-\alpha_{c_i} &\text{ if } \alpha = c_n\cdots c_{i+1} (\alpha_{c_i}), \\
	c (\alpha) &\text{ otherwise.} \\
\end{cases}
$$

\noindent
The first and the third lines of this equation are given by the root corresponding to the next occurrence of the letter associated to $\alpha$ in $\Qc$. The second line corresponds to the case when the letter associated to $\alpha$ in $\Qc$ is the last occurrence of this letter in $\Qc$. This case can be easily explained as follows. Let~$\eta :S\rightarrow S$ be the involution~$\eta(s)=\wo s \wo$. 
The last occurrence of $\eta(c_i)$ in $\Qc$ is the position which is mapped under $\tauc$ to the first occurrence of $c_i$ in $\Qc$. In other words, if we denote by $\alpha$ the root associated to the last occurrence of $\eta(c_i)$ in $\Qc$, then $\tauc(\alpha)=-\alpha_{c_i}$. In addition, the word $\woc$ is, up to commutations, equal to a word with suffix $(\eta(c_1),\dots,\eta(c_n))$~\cite[Proposition~7.1]{CeballosLabbeStump}. From this we conclude that~$\alpha=c_n\cdots c_{i+1} (\alpha_{c_i})$ as desired.
%\[\alpha=(\wo \eta(c_n)\cdots \eta(c_{i+1}))(\wo(\alpha_{c_i}))=c_n\cdots c_{i+1} (\alpha_{c_i})\] 

Using the bijection $\posToVarc$ from the positions in the word $\Qc$ to cluster variables, the rotation can also be regarded as a map on the set of cluster variables. Again for simplicity, we also write
$$\tau: \{\text{cluster variables of } \cA(W)\} \longrightarrow \{\text{cluster variables of } \cA(W)\}$$
for the composition~$\posToVarc \circ \tausqc \circ \posToVarc^{-1}$. This composition can be expressed purely in terms of cluster variables as follows. Consider the cluster variables expressed in terms of the initial cluster seed~$X_c$ associated to the weighted quiver~$\weightedQuiverc$ (recall that this quiver is by definition the Coxeter graph of the Coxeter system~$(W,S)$ directed according to the order of appearance of the simple reflections in~$c$). If~$y$ is the cluster variable at vertex~$i$ of a quiver obtained from~$\weightedQuiverc$ after a sequence of mutations~$\mu_{i_1}\rightarrow \dots \rightarrow \mu_{i_r}$, then the rotation~$\tau(y)$ is the cluster variable at vertex~$i$ of the quiver obtained from~$\weightedQuiverc$ after the sequence of mutations~$\mu_{c_1}\rightarrow \dots \rightarrow \mu_{c_n} \rightarrow \mu_{i_1}\rightarrow \dots \rightarrow \mu_{i_r}$. Although defined using a Coxeter element~$c$, this rotation map is independent of the choice of~$c$ and we denote it by~$\tau$.
\end{definition}

The following lemma is implicit in~\cite[Proposition~8.6]{CeballosLabbeStump}, see Corollary~\ref{coro:rotation}.

\begin{lemma}
The rotation map preserves clusters:
\begin{enumerate}[(i)]
\item a subset~$I\subset [m]$ of positions in the word~$\Qc$ is a $\sqc$-cluster if and only if~$\tausqc(I)$ is a $\sqc$-cluster;
\item a subset $B\subset \Phi_{\geq -1}$ of almost positive roots is a $c$-cluster if and only if~$\tauc(B)$ is a $c$-cluster; and 
\item a subset $X$ of cluster variables is a cluster if and only if~$\tau(X)$ is a cluster.
\end{enumerate}
\end{lemma}

We present a specific example for the rotation map on the positions in the word~$\Qc$, almost positive roots, and cluster variables below. For the computation of cluster variables in terms of a weighted quiver we refer the reader to~\cite{Keller}.

\begin{example}
Consider the Coxeter group~$A_2 = \fS_{3}$, generated by the simple transpositions~$s_1,s_2$ for~${s_i \eqdef (i \;\, i+1)}$, and the associated root system with simple roots $\alpha_1,\alpha_2$. 
Let~$c=s_1s_2$ be a Coxeter element and~$\Qc=(s_1,s_2,s_1,s_2,s_1)$ be the associated word.
The rotation map on the positions in the word~$\Qc$, almost positive roots, and cluster variables is given by
\renewcommand{\arraystretch}{2.1} % changes the vertical space between lines in array  
\[
\begin{tabular}{c@{ }c@{ }c@{ }c @{\qquad} c@{}c@{ }c@{ }c @{\qquad} c@{\hspace{-.3cm}}c@{ }c@{ }c}
$\tausqc:$	& $[m]$ 							& $\longrightarrow$ 		& $[m]$							&
$\tauc:$		& $\Phi_{\geq-1}$				& $\longrightarrow$		& $\Phi_{\geq-1}$				&
$\tau:$		& $\{\text{var}\}$				& $\longrightarrow$		& $\{\text{var}\}$				\\
  			& $1$   							& $\longmapsto$ 			& $3	$							&
			& $-\alpha_1$ 					& $\longmapsto$			& $\alpha_1$						&
			& $x_1$							& $\longmapsto$ 			& \bigfrac{1+x_2}{x_1}			\\
			& $2	$							& $\longmapsto$			& $4	$							&
			& $-\alpha_2$					& $\longmapsto$			& $\alpha_1 + \alpha_2$			&       
			& $x_2$							& $\longmapsto$			& \bigfrac{1+x_1+x_2}{x_1x_2}		\\
			& $3	$							& $\longmapsto$			& $5	$							&
			& $\alpha_1$						& $\longmapsto$			& $\alpha_2$						&
			& \bigfrac{1+x_2}{x_1}			& $\longmapsto$			& \bigfrac{1+x_1}{x_2}			\\
			& $4	$							& $\longmapsto$			& $1	$							&
			& $\alpha_1 + \alpha_2$			& $\longmapsto$			& $-\alpha_1$					&
			& \bigfrac{1+x_1+x_2}{x_1x_2}  	& $\longmapsto$			& $x_1$							\\
			& $5	$							& $\longmapsto$			& $2	$							&
			& $\alpha_2$						& $\longmapsto$			& $-\alpha_2$					&
			& \bigfrac{1+x_1}{x_2}			& $\longmapsto$			& $x_2$  
\end{tabular}
\]
\renewcommand{\arraystretch}{1}
\end{example} 

\begin{remark}
Let~$c$ be a bipartite Coxeter element, with sources corresponding to the positive vertices~($+$) and sinks corresponding to the negative vertices~($-$). Then, the rotation~$\tauc$ on the set of almost positive roots is the product of the maps~${\tau_+, \tau_- : \Phi_{\ge-1} \to \Phi_{\ge-1}}$ defined in~\cite[Section~2.2]{FominZelevinsky-YSystems}. We refer the interested reader to that paper for the definitions of~$\tau_+$ and~$\tau_-$.
\end{remark}

%%%%%%%%%%%%%%

\subsection{Three descriptions of~$c$-compatibility degrees}
\label{subsec:compatibilityDegrees}

In this section we introduce three notions of compatibility degrees on the set of cluster variables, almost positive roots, and positions in the word~$\Qc$. We will see in Section~\ref{sec:mainResults} that these three notions coincide under the bijections of Section~\ref{subsec:bijections}, and will use it to describe three different ways to compute $d$-vectors for cluster algebras of finite type. We refer again to \fref{fig:bijections} for a summary of our notations in these three situations.

%%%%%%

\subsubsection{On cluster variables}
Let $X=\{x_1,\dots , x_n\}$ be a set of cluster variables of~$\cA(W)$ forming a cluster, and let 
\begin{equation}
\label{eq:cluster_variable_y}
y = \frac{F(x_1,\dots,x_n)}{x_1^{d_1} \cdots x_n^{d_n}}
\end{equation}
be a cluster variable of~$\cA(W)$ expressed in terms of the variables $\{x_1,\dots ,x_n\}$ such that $F(x_1,\dots ,x_n)$ is a polynomial which is not divisible by any variable $x_j$ for $j\in [n]$.
Recall that the $d$-vector of $y$ with respect to $X$ is $\dvector(X,y) = (d_1,\dots,d_n)$. 

\begin{lemma}\label{lem:comp-denominator}
For cluster algebras of finite type, the $i$-th component of the $d$-vector $\dvector(X,y)$ is independent of the cluster $X$ containing the cluster variable $x_i$.
\end{lemma}

In view of this lemma, which is proven in Section~\ref{subsec:proof2}, one can define a compatibility degree between two cluster variables as follows.

\begin{definition}[Compatibility degree on cluster variables]
\label{def:comp-var}
For any two cluster variables~$x$ and~$y$, we denote by~$d(x,y)$ the $x$-component of the $d$-vector~$\dvector(X,y)$ for any cluster~$X$ containing the variable~$x$. We refer to~$d(x,y)$ as the \defn{compatibility degree} of~$y$ with respect to~$x$.
\end{definition}

Observe that this compatibility degree is well defined for any pair of cluster variables~$x$ and~$y$, since any cluster variable~$x$ of~$\cA(W)$ is contained in at least one cluster~$X$ of~$\cA(W)$.

%%%%%%

\subsubsection{On almost positive roots}

\begin{definition}[$c$-compatibility degree on almost positive roots]
\label{def:comp-roots}
The \defn{$c$-compatibility degree} on the set of almost positive roots is the unique function
$$
\begin{array}{ccc}
\Phi_{\geq -1} \times \Phi_{\geq -1}  & \longrightarrow  & \Z  \\
(\alpha,\beta)  & \longmapsto  & \cdeg[\alpha][\beta][c]  
\end{array}
$$
characterized by the following two properties:
\begin{eqnarray}
& \cdeg[-\alpha_i][\beta][c] = b_i, & \text{for all } i \in [n] \text{ and } \beta = \sum b_i \alpha_i \in \Phi_{\ge-1}, \label{compatibility-relation1} \\
& \cdeg[\alpha][\beta][c] = \cdeg[\tauc \alpha][ \tauc \beta][c],  & \text{for all } \alpha,\beta \in \Phi_{\ge-1}. \label{compatibility-relation2}
\end{eqnarray}
\end{definition}

\begin{remark}
\label{rem:cdeg}
This definition is motivated by the classical compatibility degree defined by S.~Fomin and A.~Zelevinsky in~\cite[Section~3.1]{FominZelevinsky-YSystems}. Namely, if $c$ is a bipartite Coxeter element, then the \mbox{$c$-compatibility} degree~$\cdeg[\cdot][\cdot][c]$ coincides with the compatibility degree~$\cdeg[\cdot][\cdot]$ of~\cite[Section~3.1]{FominZelevinsky-YSystems} except that~${\cdeg[\alpha][\alpha][c] = -1}$ while $\cdeg[\alpha][\alpha] = 0$ for any~$\alpha \in \Phi_{\ge-1}$. Throughout this paper, we ignore this difference: we still call \defn{classical compatibility degree}, and denote by~$\cdeg[\cdot][\cdot]$, the $c$-compatibility degree for a bipartite Coxeter element~$c$.
\end{remark}

\begin{remark}
In~\cite{MarshReinekeZelevinsky}, R.~Marsh, M.~Reineke, and A.~Zelevinsky defined the $c$-compatibility degree for simply-laced types in a representation theoretic way, and extended this definition for arbitrary finite type by ``folding'' techniques. In~\cite{Reading-CoxeterSortable}, N.~Reading also used the similar notion of $c$-compatibility between almost positive roots. Namely, $\alpha$ and $\beta$ are $c$-compatible when their compatibility degree vanishes. Here, we really need to know the value of the $c$-compatibility degree, and not only whether or not it vanishes.
\end{remark}

\begin{remark}
Note that it is not immediately clear from the conditions in Definition~\ref{def:comp-roots} that the $c$-compatibility degree is well-defined. Uniqueness follows from the fact that the orbits of the negative roots under~$\tauc$ cover all almost positive roots. Existence is more involved and can be proved by representation theoretic arguments. Our interpretation in Theorem~\ref{thm:three-compatibilities-coincide} below gives alternative direct definitions of $c$-compatibility, and in particular proves directly existence and uniqueness.
\end{remark}

\begin{remark}
As observed by S.~Fomin and A.~Zelevinsky in~\cite[Proposition~3.3]{FominZelevinsky-YSystems}, if~$\alpha, \beta$ are two almost positive roots and~$\alpha^\vee,\beta^\vee$ are their dual roots in the dual root system, then ${\cdeg[\alpha][\beta] = \cdeg[\beta^\vee][\alpha^\vee]}$. Although not needed in this paper, we remark that this property also holds for $c$-compatibility degrees. 
%This property is for example illustrated by the fact that the $c$-compatibility tables~\ref{table:compB} and~\ref{table:compC} of respective types~$B_2$ and~$C_2$ are transpose to each other.
\end{remark}

%%%%%%

\subsubsection{On positions in the word $\Qc$}
\label{subsec:rootFunction}

In this section, we recall the notion of root functions associated to $\sqc$-clusters in~$\Qc$, and use them in order to define a $\sqc$-compatibility degree on the set of positions in~$\Qc$. This description relies only on linear algebra and is one of the main contributions of this paper. The root function was defined by C.~Ceballos, J.-P.~Labb\'e, and C.~Stump in~\cite[Definition~3.2]{CeballosLabbeStump} and was extensively used by V.~Pilaud and C.~Stump in the construction of Coxeter brick polytopes~\cite{PilaudStump-brickPolytopes}.

\begin{definition}[\cite{CeballosLabbeStump}]
\label{def:rootFunction}
The \defn{root function}
$$\Root{I}{\cdot} : [m] \longrightarrow \Phi$$
associated to a $\sqc$-cluster $I \subseteq [m]$ in~$\Qc$ is defined by
$$\Root{I}{j} \eqdef \wordprod{\sqc}{[j-1]\ssm I}(\alpha_{q_j}),$$
where $\wordprod{\sqc}{X}$ denotes the product of the reflections $q_x \in \Qc$ for $x \in X$ in this order. The \defn{root configuration} of~$I$ is the set $\Roots{I} \eqdef \set{\Root{I}{i}}{i \in I}$ (Although the root configuration is a priori a multi-set for general subword complexes, it is indeed a set with no repeated elements for the particular choice of word $\Qc$.)
\end{definition} 

As proved in~\cite[Section~3.1]{CeballosLabbeStump}, the root function~$\Root{I}{\cdot}$ encodes exchanges in the $\sqc$-cluster~$I$. Namely, any~$i \in I$ can be exchanged with the unique~$j \notin I$ such that~${\Root{I}{j} = \pm \Root{I}{j}}$ (see Lemma~\ref{lem:cluster_flip}), and the root function can be updated during this exchange (see Lemma~\ref{lem:cluster_update}). It was moreover shown in~\cite[Section~6]{PilaudStump-brickPolytopes} that the root configuration~$\Roots{I}$ forms a basis for~$\R^n$ for any given initial $\sqc$-cluster~$I$ in~$\Qc$. It enables us to decompose any other root on this basis to get the following coefficients, which will play a central role in the remainder of the paper.

\begin{definition}[$\sqc$-compatibility degree on positions in $\Qc$]
\label{def:comp-pos}
Fix any initial $\sqc$-cluster~$I \subseteq [m]$ of~$\Qc$.
For any position~$j \in [m]$, we decompose the root $\Root{I}{j}$ on the basis~$\Roots{I}$ as follows:
$$\Root{I}{j} = \sum_{i\in I} \rho_i(j) \Root{I}{i}.$$
For~$i \in I$ and~$j \in [m]$, we define the \defn{$\sqc$-compatibility degree} as the
coefficient
$$ \vcoeff[i][j][\sqc] =
\begin{cases}
	\rho_i(j) &\text{ if } j > i, \\
	-\rho_i(j) &\text{ if } j \leq i. \\
\end{cases}
$$
\end{definition}

According to the following lemma, it is indeed valid to omit to mention the specific $\sqc$-cluster~$I$ in which these coefficients are computed. 
We refer to Section~\ref{subsec:proof1} for the proof of Lemma~\ref{lem:indep1}.

\begin{lemma}
\label{lem:indep1}
The coefficients $\vcoeff[i][j][\sqc]$ are independent of the choice of the $\sqc$-cluster $I \subseteq [m]$ of~$\Qc$ containing~$i$.
\end{lemma}

Moreover, this compatibility degree is well defined for any pair of positions~${i,j \in [m]}$ of the word $\Qc$, since for any~$i$ there is always a~$\sqc$-cluster~$I$ containing~$i$.

%%%%%%%%%%%%%%%%%%%%%%%%%%%%%%%%%%%%%%

\section{Main results: Three descriptions of $d$-vectors}
\label{sec:mainResults}

In this section we present the main results of this paper. 

\begin{theorem}\label{thm:three-compatibilities-coincide}
The three notions of compatibility degrees on the set of cluster variables, almost positive roots, and positions in the word $\Qc$ coincide under the bijections of Section~\ref{subsec:bijections}. More precisely, for every pair of positions~$i,j$ in the word $\Qc$ we have
$$d(\posToVarc(i), \posToVarc(j)) = \cdeg[\posToRootsc(i)][\posToRootsc(j)][c] = \vcoeff[i][j][\sqc].$$
In particular, if $c$ is a bipartite Coxeter element, then these coefficients coincide with the classical compatibility degrees of S.~Fomin and A.~Zelevinsky~\cite[Section~3.1]{FominZelevinsky-YSystems} (except for Remark~\ref{rem:cdeg}).
\end{theorem}

The proof of this theorem can be found in Section~\ref{sec:main-proof}.
The following three statements are the main results of this paper and are direct consequences of Theorem~\ref{thm:three-compatibilities-coincide}. 
%%%
The first statement describes the denominator vectors in terms of the compatibility degrees of~\cite[Section~3.1]{FominZelevinsky-YSystems}.

\begin{corollary}
\label{thm:d_vectors-compatibilty}
Let~$B \eqdef \{\beta_1, \dots, \beta_n\} \subseteq \Phi_{\ge-1}$ be a (classical) cluster in the sense of S.~Fomin and A.~Zelevinsky~\cite[Theorem~1.9]{FominZelevinsky-ClusterAlgebrasII}, and let~$\beta \in \Phi_{\ge-1}$ be an almost positive root. Then the $d$-vector~$\dvector(B,\beta)$ of the cluster variable~$\rootsToVar(\beta)$ with respect to the initial cluster seed~$\rootsToVar(B) = \{\rootsToVar(\beta_1), \dots, \rootsToVar(\beta_n)\}$ is given by
$$\dvector(B,\beta) = \big( \cdeg[\beta_1][\beta], \dots , \cdeg[\beta_n][\beta] \big),$$
where $\cdeg[\beta_i][\beta]$ is the compatibility degree of $\beta$ with respect to $\beta_i$ as defined by S.~Fomin and A.~Zelevinsky~\cite[Section~3.1]{FominZelevinsky-YSystems} (except for Remark~\ref{rem:cdeg}).
\end{corollary}

%%%

The next statement extends this result to any Coxeter element~$c$ of~$W$.

\begin{corollary}
\label{thm:d_vectors-c_compatibilty}
Let~$B \eqdef \{\beta_1, \dots, \beta_n\} \subseteq \Phi_{\ge-1}$ be a $c$-cluster in the sense of N.~Reading~\cite[Section~ 7]{Reading-CoxeterSortable}, and let~$\beta \in \Phi_{\ge-1}$ be an almost positive root. Then the $d$-vector~$\dvector_c(B,\beta)$ of the cluster variable~$\rootsToVarc(\beta)$ with respect to the initial cluster seed~$\rootsToVarc(B) = \{\rootsToVarc(\beta_1), \dots, \rootsToVarc(\beta_n)\}$ is given by
$$\dvector_c(B,\beta) = \big( \cdeg[\beta_1][\beta][c], \dots , \cdeg[\beta_n][\beta][c] \big),$$
where $\cdeg[\beta_i][\beta][c]$ is the $c$-compatibility degree of $\beta$ with respect to $\beta_i$ as defined in Definition~\ref{def:comp-roots}.
\end{corollary}

%%%

Finally, the third statement describes the denominator vectors in terms of the coefficients~$\vcoeff[i][j][\sqc]$ obtained from the word $\Qc$.

\begin{corollary}
\label{thm:d-vector}
Let~$I \subseteq [m]$ be a $\sqc$-cluster and~$j \in [m]$ be a position in~$\Qc$.
Then the $d$-vector~$\dvector_\sqc(I,j)$ of the cluster variable~$\posToVarc(j)$ with respect to the initial cluster seed~$\posToVarc(I) = \set{\posToVarc(i)}{i \in I}$ is given by
$$\dvector_\sqc(I,j) = \big( \vcoeff[i][j][\sqc] \big)_{i\in I}.$$
\end{corollary}

As a consequence, we obtain the following result which is proven in Section~\ref{sec:proof-coro-positive}.  

\begin{corollary}
\label{coro:positive}
For cluster algebras of finite type, the $d$-vector of a cluster variable that is not in the initial seed is non-negative and not equal to zero.
\end{corollary}

This corollary was conjectured by S.~Fomin and A.~Zelevinsky for arbitrary cluster algebras~\cite[Conjecture~7.4]{FominZelevinsky-ClusterAlgebrasIV}. In the case of cluster algebras of finite type, this conjecture also follows from~\cite[Theorem~4.4 and Remark~4.5]{CalderoChapotonSchiffler} and from~\cite[Theorem~2.2]{BuanMarshReiten}, where the authors show that the $d$-vectors can be computed as the dimension vectors of certain indecomposable modules.

%%%%%%%%%%%%%%%%%%%%%%%%%%%%%%%%%%%%%%

\section{Cluster variables and orbits of positions in the bi-infinite word~$\tQ$}
\label{subsec:mobius}

We now want to present the results of this paper in terms of orbits of letters of the bi-infinite word~$\tQ \eqdef (\tsqq_i)_{i \in \Z}$ obtained by infinitely many repetitions of the product of all generators in~$S$. Up to commutations of consecutive commuting letters, this word does not depend on the order of the elements of~$S$ in this product.  Our motivation is to avoid the dependence in the Coxeter element~$c$, which is just a technical tool to deal with clusters in terms of almost positive roots or positions in the word~$\Qc$. This point of view was already considered in type~$A$ by V.~Pilaud and M.~Pocchiola in~\cite{PilaudPocchiola} in their study of pseudoline arrangements on the M\"obius strip. For arbitrary finite Coxeter groups, the good formalism is given by the Auslander-Reiten quiver (see \eg~\cite[Section~8]{CeballosLabbeStump}). However, we do not really need this formalism here and skip its presentation. The proofs of this section are omitted but can be easily deduced from the results proved in Section~\ref{sec:main-proof} (more precisely from Lemmas~\ref{lem:rotation1}, \ref{lem:cluster-jumping} and~\ref{lem:compatibility-rotation}) and Corollary~\ref{thm:d-vector}.

We denote by~$\eta :S\rightarrow S$ the involution~$\eta(s)=\wo s \wo$ which conjugates a simple reflection by the longest element~$\wo$ of~$W$, and by~$\eta(\Qc)$ the word obtained by conjugating each letter of~$\Qc$ by~$\wo$. As observed in~\cite{CeballosLabbeStump}, the bi-infinite word~$\tQ$ coincides, up to commutations of consecutive commuting letters, with the bi-infinite word~$\cdots \Qc \, \eta(\Qc) \, \Qc \, \eta(\Qc) \, \Qc \cdots$ obtained by repeating infinitely many copies of~$\Qc \, \eta(\Qc)$. We can therefore consider the word~$\Qc$ as a fundamental domain in the bi-infinite word~$\tQ$, and a position~$i$ in~$\Qc$ as a representative of an orbit~$\ti$ of positions in~$\tQ$ under the translation map~$\translation : i \mapsto i+m$ (note that this translation maps the letter~$q_i$ to the conjugate letter~$q_{i+m} = \wo q_i \wo$). For a subset~$I$ of positions in~$\Qc$, we denote by~$\tI \eqdef \set{\ti}{i \in I}$ its corresponding orbit in~$\tQ$. It turns out that the orbits of $\sqc$-clusters are now independent of~$\sqc$.

\begin{proposition}
Let~$\sqc$ and~$\sqc'$ be reduced expressions of two Coxeter elements. Let~$I$ and~$I'$ be subsets of positions in~$\Qc$ and~$\Q_{\sqc'}$ respectively such that their orbits~$\tI$ and~$\tI'$ in~$\tQ$ coincide. Then~$I$ is a $\sqc$-cluster in~$\Qc$ if and only if~$I'$ is a $\sqc'$-cluster of~$\Q_{\sqc'}$. We then say that~$\tI = \tI'$ forms a cluster in~$\tQ$.
\end{proposition}

In other words, we obtain a bijection~$\orbToVar$ from the orbits of positions in~$\tQ$ (under the translation map~${\translation : i \mapsto i+m}$) to the cluster variables of~$\cA(W)$. A collection of orbits forms a cluster if and only if their representatives in any (or equivalently all) fundamental domain~$\Qc$ for~$\translation$ form the complement of a reduced expression for~$\wo$. Choosing a particular Coxeter element~$c$ defines a specific fundamental domain~$\Qc$ in~$\tQ$, which provides specific bijections~$\posToRootsc$ and ~$\rootsToVarc$ with almost positive roots and cluster variables. We insist on the fact that~$\orbToVar$ does not depend on the choice of a Coxeter element, while~$\posToRootsc$ and ~$\rootsToVarc$ do.

We now want to describe the results of this paper directly on the bi-infinite word~$\tQ$. We first transport the $d$-vectors through the bijection~$\orbToVar$: for a given initial cluster seed~$\tI$ in~$\tQ$, and an orbit~$\tj$ of positions in~$\tQ$, we define the $d$-vector
$$\dvector(\tI,\tj) \eqdef \dvector \big( \orbToVar(\tI), \orbToVar(\tj) \big).$$
We want to express these $d$-vectors in terms of the coefficients~$\vcoeff[i][j][\sqc]$ from Definition~\ref{def:comp-pos}. For this, we first check that these coefficients are independent of the fundamental domain~$\Qc$ on which they are computed.

\begin{proposition}
\label{prop:coeffs}
Let~$\sqc$ and~$\sqc'$ be reduced expressions of two Coxeter elements.
Let~$i,j$ be positions in~$\Qc$ and $i',j'$ be positions in~$\Q_{\sqc'}$ be such~${\ti = \ti'}$ and~${\tj = \tj'}$.
Then the coefficients~$\vcoeff[i][j][\sqc]$ and~$\vcoeff[i'][j'][\sqc']$, computed in~$\Qc$ and~$\Q_{\sqc'}$ respectively, coincide.
\end{proposition}

For any orbits~$\ti$ and~$\tj$ of~$\tQ$, we can therefore define with no ambiguity the coefficient~$\vcoeff[\ti][\tj]$ to be the coefficient~$\vcoeff[i][j][\sqc]$ for~$i$ and~$j$ representatives of~$\ti$ and~$\tj$ in any arbitrary fundamental domain~$\Qc$. The $d$-vectors of the cluster algebra can then be expressed from these coefficients.

\begin{theorem}
Let~$\tI$ be a collection of orbits of positions in~$\tQ$ forming a cluster in~$\tQ$, and let~$\tj$ be an orbit of positions in~$\tQ$. Then the $d$-vector~$\dvector(\tI,\tj)$ of the cluster variable~$\orbToVar(\tj)$ with respect to the initial cluster seed~$\orbToVar(\tI)$ is given by
$$\dvector(\tI,\tj) = (\vcoeff[\ti][\tj])_{\ti \in \tI}.$$
\end{theorem}

%%%%%%%%%%%%%%%%%%%%%%%%%%%%%%%%%%%%%%

\section{Proofs of Lemma~\ref{lem:comp-denominator} and Lemma~\ref{lem:indep1}}
\label{sec:proof}

%%%%%%%%%%%%%%

\subsection{Proof of Lemma~\ref{lem:comp-denominator}}
\label{subsec:proof2}

For cluster algebras of finite type all clusters containing the cluster variable $x_i$ are connected under mutations. Therefore, it is enough to prove the lemma for the cluster 
$X'=\{x_1,\dots , x_{j-1} , x_j', x_{j+1},\dots , x_n\}$  obtained from $X=\{x_1,\dots ,x_n\}$ by mutating a variable $x_j$ with $j\neq i$. The variables $x_j$ and $x_j'$ satisfy a relation 
$$ x_j = \frac{P(x_1, \dots , \widehat x_j, \dots , x_n)}{x_j'}$$
for a polynomial $P$ in the variables $\{x_1, \dots , \widehat x_j, \dots , x_n \}$, where~$\widehat x_j$ means that we skip variable~$x_j$. Replacing $x_j$ in equation~\eqref{eq:cluster_variable_y} we obtain 
$$ y = \frac{\widetilde P(x_1,\dots, x_j' ,\dots,,x_n)  /  {x_j'}^{m_j}}
{x_1^{d_1} \cdots \widehat {x_j^{d_j}} \cdots x_n^{d_n}  P(x_1, \dots , \widehat x_j, \dots , x_n)   /  {x_j'}^{d_j} } $$
where $m_j$ is a non-negative integer and $\widetilde P$ is a polynomial which is not divisible by $x_j'$ and by any $x_\ell$ with $\ell\neq j$.
The Laurent phenomenon implies that $P$ divides $\widetilde P$. Thus , we obtain 
$$ y = \frac{\widetilde F(x_1,\dots, x_j' , \dots ,x_n)}{x_1^{d_1} \cdots \widehat {x_j^{d_j}} \cdots x_n^{d_n}} \cdot
\frac{{x_j'}^{d_j}}{{x_j'}^{m_j}}
$$
is the rational expression of the cluster variable $y$ expressed in terms of the variables of $X'$.
As a consequence, the $d$-vectors $\dvector(X,y)$ and $\dvector(X',y)$ differ only in the $j$-th coordinate. In particular, the $i$-th coordinate remains constant after mutation of any $j\neq i$ as desired.

%%%%%%%%%%%%%%

\subsection{Proof of Lemma~\ref{lem:indep1}}
\label{subsec:proof1}

Before proving this lemma we need some preliminaries on subword complexes.
Recall that $I\subset [m]$ is a $\sqc$-cluster of~$\Qc$ if and only if the subword of $\Qc$ with positions at the complement of $I$ is a reduced expression of $\wo$ (see Theorem~\ref{thm:CLS_cluster_complexes}). 

\begin{lemma}[{\cite[Lemma~3.3]{CeballosLabbeStump}}]
\label{lem:cluster_flip}
Let $I \subset [m]$ be a $\sqc$-cluster of $\Qc$. Then, 
\begin{enumerate}[(i)]
\item For every $i\in I$ there exist a unique $j\notin I$ such that $I \symdif \{i,j\}$ is again a $\sqc$-cluster, where~$A \symdif B \eqdef (A \cup B) \ssm (A \cap B)$ denotes the symmetric difference.
\item This $j$ is the unique $j\notin I$ satisfying $\Root{I}{j}=\pm \Root{I}{i}$.
\end{enumerate}
\end{lemma}

This exchange operation between $\sqc$-clusters is called \defn{flip}. It correspond to mutations between clusters in the cluster algebra. During the flip, the root function is updated as follows.

\begin{lemma}[{\cite[Lemma~3.6]{CeballosLabbeStump}}]
\label{lem:cluster_update}
Let $I$ and $J$ be two adjacent $\sqc$-clusters of $\Qc$ with~$I \ssm i = J \ssm j$, and assume that $i<j$. Then, for every $k\in [m]$,
$$
\Root{I'}{k}=
\begin{cases}
	t_i(\Root{I}{k}) &\text{ if } i<k\leq j, \\
	\Root{I}{k} &\text{ otherwise } \\
\end{cases}
$$
Here, $t_i = w q_iw^{-1}$ where $w$ is the product of the reflections $q_x\in \Qc$ for $x\in [i-1]\ssm I$. 
By construction, $t_i$ is the reflection in $W$ orthogonal to the root $\Root{I}{i}=w(\alpha_{q_i})$.
\end{lemma}

This result implies the following Lemma.

\begin{lemma} \label{lem:flip_root_function}
Let $I$ and $J$ be two adjacent $\sqc$-clusters of $\Qc$ with~$I \ssm i = J \ssm j$. Then, for every $k \in [m]$,
$$\Root{J}{k}= \Root{I}{k} + a_k \Root{I}{i}$$
for some constant~$a_k \in \R$.
\end{lemma}

Using the previous lemma, we now derive Lemma~\ref{lem:indep1}.

\begin{proof}[Proof of Lemma~\ref{lem:indep1}]
This result is equivalent to prove that the coefficients $\rho_i(j)$ from Definition~\ref{def:comp-pos} are independent of the $\sqc$-cluster $I\subset [m]$ of $\Qc$ containing $i$. Since all the $\sqc$-clusters containing~$i$ are connected by flips, it is enough to prove that $\rho_i(j)$ is preserved by any flip not involving $i$. Let $i\in I$ and let $i' \in I \ssm i$. Then, $i'$ can be exchanged with a unique $j'\notin I$ such that $I'=I \symdif \{i',j'\}$ is again a $\sqc$-cluster. By Lemma~\ref{lem:flip_root_function} and part~(ii) of Lemma~\ref{lem:cluster_flip}, 
$$\Root{I}{j} = \sum_{\ell \in I} \rho_\ell(j) \Root{I}{\ell}.$$
implies
$$\Root{I'}{j} = \sum_{\ell \in I \ssm i'} \rho_\ell(j) \Root{I'}{\ell} +  a \, \Root{I}{j'}$$
for some constant~$a\in \R$. In particular, this implies that the coefficients~$\rho_i(j)$ are the same for~$I$ and~$I'$.
\end{proof}

%%%%%%%%%%%%%%%%%%%%%%%%%%%%%%%%%%%%%%

%\newpage
\section{Proof of Theorem~\ref{thm:three-compatibilities-coincide}}
\label{sec:main-proof}

Our proof of Theorem~\ref{thm:three-compatibilities-coincide} is based in Proposition~\ref{coro:compatibility-denom} and Proposition~\ref{coro:compatibility-coef}. These propositions are stated and proved in Section~\ref{sec:mainproofpart1} and Section~\ref{sec:mainproofpart2} respectively. 

%%%%%%%%%%%%%%

\subsection{The map~$d(\cdot,\cdot)$ satisfies relations~\eqref{compatibility-relation1} and~\eqref{compatibility-relation2}}
\label{sec:mainproofpart1}

In this section we show that the map~$d(\cdot,\cdot)$ induces a map on the set of almost positive roots which satisfy the properties~\eqref{compatibility-relation1} and~\eqref{compatibility-relation2} in the definition of the $c$-compatibility degree among almost positive roots (Definition~\ref{def:comp-roots}). Before stating this result in Proposition~\ref{coro:compatibility-denom} we need the following lemma.

\begin{lemma}
\label{lem:var-rotation}
Let $X=\{x_1,\dots ,x_n\}$ be a cluster and let $y$ be a cluster variable of~$\cA$. Then the rational function of $y$ expressed in terms of the variables in $X$ is exactly the same as the rational function of $\tau y$ in terms of the variables in $\tau X$. In particular, 
$$ \dvector(X,y) = \dvector(\tau X, \tau y).$$
\end{lemma}

\begin{proof}
We prove this proposition in two parts: first in the case where the cluster seed $X$ corresponds to a Coxeter element $c$, and then for any arbitrary cluster~seed.

Let~$X_c$ be the cluster corresponding to a Coxeter element $c$, \ie the set of variables on the vertices of the weighted quiver~$\weightedQuiverc$ corresponding to~$c$. By definition of the rotation map on the set of cluster variables, the rotated cluster~$\tau X_c$ consists of the variables on the vertices of the quiver obtained from~$\weightedQuiverc$ by consecutively applying the mutations $\mu_{c_1}\rightarrow \dots \rightarrow \mu_{c_n}$. The resulting underlying quiver $\mathcal Q'$ after these mutations is exactly equal to~$\weightedQuiverc$. 
Moreover, every sequence of mutations giving rise to a variable~$y$, which starts at the cluster seed~$(\weightedQuiverc, X_c)$ can be viewed as a sequence of mutations giving rise to~$\tau y$ starting at the rotated seed cluster~$(\mathcal Q', \tau X_c)$.
Since the quivers~$\weightedQuiverc$ and~$\mathcal Q'$ coincide, the rational functions for $y$ and $\tau y$ in terms of the variables of $X_c$ and $\tau X_c$ respectively are exactly the same as desired.

In the general case where $X$ is an arbitrary cluster seed we proceed as follows. For any cluster variable $y$ and any clusters $X$ and $Y$, denote by $y(X)$ the rational function of $y$ in terms of the variables of the cluster $X$, and by $Y(X)$ the rational functions of the variables of $Y$ with respect to the variables of the cluster $X$. Then, using the fact $y(X)=y(Y)\circ Y(X)$ and the first part of this proof we obtain
$$y(X) = y(X_c) \circ X_c(X) = \tau y(\tau X_c) \circ \tau X_c(\tau X) = \tau y(\tau X)$$
as desired.
\end{proof}

As a direct consequence of this lemma we obtain the following key result.

\begin{proposition}
\label{coro:compatibility-denom}
Let~$c$ be a Coxeter element. For every position~$i$ in the word~$\Qc$ denote by~$x_i = \posToVarc(i)$ the associated cluster variable. 
Then, the values~$d(x_i,x_j)$ satisfy the following two properties:
\begin{eqnarray}
& d(x_i,x_j) = b_i, & \text{for all } i \in [n] \text{ and } \posToRootsc (j) = \sum b_i \alpha_i \in \Phi_{\ge-1},  \label{eq:comp-den1} \\
& d(x_i,x_j) = d(\tau x_i, \tau x_j), & \text{for all } i,j \in [m]. \label{eq:comp-den2}
\end{eqnarray}
\end{proposition}

\begin{proof}
Let~$X_c$ be the set of cluster variables corresponding to the first~$n$ positions of the word~$\Qc$, or equivalently, the cluster associated to the weighted quiver~$\weightedQuiverc$ corresponding to the Coxeter element~$c$. 
As discussed in Section~\ref{subsec:bijections}, the $d$-vector of the variable~$x_j=\rootsToVarc (\posToRootsc (j) )$ in terms of this initial quiver~$\weightedQuiverc$ is given by the almost positive root~$\posToRootsc(j)$. More precisely, we have
$$\dvector(X_c,x_j) = (b_1,\dots ,b_n).$$
This implies relation~\eqref{eq:comp-den1} of the proposition. Relation~\eqref{eq:comp-den2} follows directly from Lemma~\ref{lem:var-rotation}.
\end{proof}

%%%%%%%%%%%%%%

\subsection{The map~$\vcoeff[\cdot][\cdot][\sqc]$ satisfies relations~\eqref{compatibility-relation1} and~\eqref{compatibility-relation2}}
\label{sec:mainproofpart2}

In this section we show that the map~$\vcoeff[\cdot][\cdot][\sqc]$ induces a map on the set of almost positive roots which satisfy the properties~\eqref{compatibility-relation1} and~\eqref{compatibility-relation2} in the definition of the $c$-compatibility degree among almost positive roots (Definition~\ref{def:comp-roots}). Before stating this result in Proposition~\ref{coro:compatibility-coef} we need some preliminaries concerning the coefficients~$\vcoeff[i][j][\sqc]$. 

For this, we use an operation of jumping letters between the words~$\Qc$ defined in~\cite[Section 3.2]{CeballosLabbeStump}. We denote by~$\eta :S\rightarrow S$ the involution~$\eta(s)=\wo s \wo$ which conjugates a simple reflection by the longest element~$\wo$ of~$W$. Given a word~$\Q \eqdef (q_1, q_2, \cdots, q_r)$, we say that the word~$(q_2, \cdots, q_r, \eta(q_1))$ is the \defn{jumping word} of~$\Q$, or is obtained by \defn{jumping} the first letter in~$\Q$. In the following three lemmas, we consider a reduced expression~$\sqc \eqdef (s,c_2,\dots,c_n)$ of a Coxeter element~$c$, and the reduced expression~$\sqc' \eqdef (c_2,\dots,c_n,s)$ of the Coxeter element~$c' \eqdef scs$ obtained by deleting the letter~$s$ in~$\sqc$ and putting it at the end.

\begin{lemma}[{\cite[Proposition~4.3]{CeballosLabbeStump}}]
\label{lem:rotation1}
The jumping word of~$\Qc$ coincides with the word~$\Q_{\sqc'}$ up to commutations of consecutive commuting letters.
\end{lemma}

Let~$\sigma$ denote the map from positions in~$\Qc$ to positions in~$\Q_{\sqc'}$ which jumps the first letter and reorders the letters (by commutations).

\begin{lemma}[{\cite[Proposition~3.9]{CeballosLabbeStump}}]
\label{lem:cluster-jumping}
A subset~$I\subset [m]$ of positions in the word~$\Qc$ is a $\sqc$-cluster if and only if~$\sigma(I)$ is a $\sqc'$-cluster in~$\Q_{\sqc'}$.
\end{lemma}

We include a short proof of this lemma here for convenience of the reader.

\begin{proof}
Let~$P$ be the subword of~$\Qc$ with positions at the complement of~$I$ in~$[m]$, and let~$P'$ be the subword of~$\Q_{\sqc'}$ with positions at the complement of~$\sigma(I)$ in~$[m]$. 
Recall from Theorem~\ref{thm:CLS_cluster_complexes} that~$I$ is a $\sqc$-cluster in~$\Qc$ if and only if~$P$ is a reduced expression of~$\wo$. If~$1\in I$, then~$P$ and~$P'$ are the same. If~$1\notin I$, then~$P'$ is the jumping word of~$P$. In both cases~$P$ is a reduced expression of~$\wo$ if and only if~$P'$ is a reduced expression of~$\wo$. Therefore, $I$ is a $\sqc$-cluster in~$\Qc$ if and only if~$\sigma(I)$ is a $\sqc'$-cluster in~$\Q_{\sqc'}$.
\end{proof}

Observe now that we obtain again~$\sqc$ when we jump repeatedly all its letters. However, a position~$i$ in~$\Qc$ is rotated to the position~$\tausqc^{-1}(i)$ by these operations. This implies the following statement.

\begin{corollary}
\label{coro:rotation}
A subset~$I\subset [m]$ of positions in the word~$\Qc$ is a $\sqc$-cluster if and only if~$\tausqc(I)$ is a $\sqc$-cluster.
\end{corollary}

Using the jumping operation studied above, we now derive the following results concerning the coefficients~$\vcoeff[i][j][\sqc]$.

\begin{lemma}
\label{lem:compatibility-rotation}
If~$i' \eqdef \sigma(i)$ and $j' \eqdef \sigma(j)$ denote the positions in~$\Q_{\sqc'}$ corresponding to positions $i$ and $j$ in~$\Qc$ after jumping the letter~$s$, then
$$\vcoeff[i'][j'][\sqc'] = \vcoeff[i][j][\sqc].$$
\end{lemma}

\begin{proof}
Let~$I$ be a~$\sqc$-cluster and~$k$ be a position in~$\Qc$. We denote by~$I' \eqdef \sigma(I)$ the $\sqc'$-cluster corresponding to~$I$ and by~$k' \eqdef \sigma(k)$ the position in~$\Q_{\sqc'}$ corresponding to~$k$ after jumping the letter~$s$. By the definition of the root function we obtain that
$$\Root{I'}{k'} = \begin{cases} -\Root{I}{k} & \text{if } k = 1 \in I, \\ \phantom{-}\Root{I}{k} & \text{if } k = 1 \notin I, \\ \phantom{-}\Root{I}{k} & \text{if } k \ne 1 \in I, \\ \phantom{.}s\Root{I}{k} & \text{if } k \ne 1 \notin I, \end{cases}$$
Applying this relation to a $\sqc$-cluster~$I$ containing~$i$ and a position~$j\neq i$, we derive that
$$\rho_{i'}(j') = \begin{cases} - \rho_i(j) & \text{if } i = 1 \text{ or } j = 1, \\ \phantom{-}\rho_i(j) & \text{otherwise,}\end{cases}$$
where~$\rho_i(j)$ denotes the $i$th coordinate of~$\Root{I}{j}$ in the linear basis~$\Roots{I}$ and~$\rho_{i'}(j')$ denotes the $i'$th coordinate of~$\Root{I'}{j'}$ in the linear basis~$\Roots{I'}$. This implies that $\vcoeff[i'][j'][\sqc'] = \vcoeff[i][j][\sqc]$ by definition of~$\vcoeff[\cdot][\cdot][\sqc]$.
\end{proof}

\begin{proposition}
\label{coro:compatibility-coef}
Let~$c$ be a Coxeter element and~$i$ and~$j$ be positions in the word~$\Qc$. Then, the coefficients $\vcoeff[i][j][\sqc]$ satisfy the following two properties:
\begin{eqnarray}
&\vcoeff[i][j][\sqc] = b_i, & \text{for all } i \in [n] \text{ and } \posToRootsc (j) = \sum b_i \alpha_i \in \Phi_{\ge-1}, \label{eq:comp-coef1} \\
& \vcoeff[i][j][\sqc] = \vcoeff[\tausqc i][\tausqc j][\sqc], & \text{for all } i,j \in [m]. \label{eq:comp-coef2}
\end{eqnarray}
\end{proposition}

\begin{proof}
Let~$I_\sqc=[n]$ be the $\sqc$-cluster given by the first~$n$ positions in the word~$\Qc$. Then
$$
\Root{I_\sqc}{j} = 
\begin{cases} 
\alpha_{c_j} & \text{if } 1\leq j \leq n, \\ 
\posToRootsc(j) & \text{if } n < j \leq m.
\end{cases}
$$ 
Therefore, by definition of the coefficients~$\vcoeff[i][j][\sqc]$ we have that the vector~$(\vcoeff[i][j][\sqc])_{i\in I_\sqc}$ is given by the almost positive root~$\posToRootsc(j)$. More precisely, 
$$(\vcoeff[i][j][\sqc])_{i\in I_\sqc} = (b_1, \dots , b_n).$$
This implies relation~\eqref{eq:comp-coef1} of the proposition. Relation~\eqref{eq:comp-coef2} follows from jumping repeatedly all the letters of~$\sqc$ in~$\Qc$. Thus, position~$i$ in~$\Qc$ is rotated to position~$\tausqc^{-1}(i)$ by these operations, and the result follows from~$n$ applications of Lemma~\ref{lem:compatibility-rotation}. 
\end{proof}

%%%%%%%%%%%%%%

\subsection{Proof of Theorem~\ref{thm:three-compatibilities-coincide}}

By Definition~\ref{def:comp-roots}, Proposition~\ref{coro:compatibility-denom} and Proposition~\ref{coro:compatibility-coef}, the three maps~$d( \posToVarc(\cdot) , \posToVarc(\cdot) )$, $\cdeg[\posToRootsc(\cdot)] [\posToRootsc(\cdot)] [c]$ and $\vcoeff[\cdot][\cdot][\sqc]$
on the set of positions in the word~$\Qc$ satisfy the two properties~\eqref{eq:comp-coef1} and~\eqref{eq:comp-coef2} of Proposition~\ref{coro:compatibility-coef}. 
Since these properties uniquely determine a map on the positions in~$\Qc$ the result follows.

%%%%%%%%%%%%%%

\subsection{Proof of Corollary~\ref{coro:positive}}
\label{sec:proof-coro-positive}

In this section we present two independent proofs of Corollary~\ref{coro:positive}. The first proof follows from the description of $d$-vectors in terms of $c$-compatibility degrees~$\cdeg[\alpha][\beta][c]$ between almost positive roots as presented in Corollaries~\ref{thm:d_vectors-compatibilty} and~\ref{thm:d_vectors-c_compatibilty}. The second proof is based on the description of $d$-vectors in terms of the coefficients~$\vcoeff[i][j][\sqc]$ obtained from the word~$\Qc$ as presented in Corollary~\ref{thm:d-vector}. Our motivation for including this second proof here is to extend several results of this paper to the family of ``root-independent subword complexes". This family of simplicial complexes, defined in~\cite{PilaudStump-brickPolytopes}, contains all cluster complexes of finite type. Using the brick polytope approach~\cite{PilaudSantos-brickPolytope, PilaudStump-brickPolytopes}, V.~Pilaud and C.~Stump constructed polytopal realizations of these simplicial complexes. Extending the results of this paper to root-independent subword complexes might lead to different polytopal realizations of these simplicial complexes.

%%%%%%%

\subsubsection{First proof}

Corollary~\ref{coro:positive} follows from the following known fact.

\begin{lemma}[\cite{FominZelevinsky-YSystems, MarshReinekeZelevinsky, Reading-CoxeterSortable}]
\label{lem:non-neg1}
The compatibility degree~$\cdeg[\alpha][\beta][c]$ is non-negative for any pair of almost positive roots~$\alpha \neq \beta$. Moreover,~$\cdeg[\alpha][\beta][c]=0$ if and only if~$\alpha$ and~$\beta$ are $c$-compatible, \ie if they belong to some $c$-cluster.
\end{lemma}

Let~$B=\{\beta_1,\dots,\beta_n\} \subset \Phi_{\geq -1}$ be an initial $c$-cluster seed and~$\beta\in \Phi_{\geq-1}$ be an almost positive root which is not in~$B$. Then, by Lemma~\ref{lem:non-neg1} there is at least one~$i\in [n]$ such that~$\cdeg[\beta_i][\beta][c]>0$, otherwise~$\beta$ would be $c$-compatible to with all~$\beta_i$ which is a contradiction. 
Corollary~\ref{coro:positive} thus follows from Corollary~\ref{thm:d_vectors-c_compatibilty} and Lemma~\ref{lem:non-neg1}.

%%%%%%%

\subsubsection{Second proof}
Corollary~\ref{coro:positive} follows from the next statement.

\begin{lemma}
The coefficient $\vcoeff[i][j][\sqc]$ is non-negative for any pair of positions $i\neq j$ in the word~$\Qc$. Moreover, $\vcoeff[i][j][\sqc] = 0$ if and only if $i$ and $j$ are compatible, \ie if they belong to some $\sqc$-cluster.
\end{lemma}

\begin{proof}
The non-negativity is clear if~$i$ is one of the first~$n$ letters. Indeed, computing in the $\sqc$-cluster~$I$ given by the initial prefix~$\sqc$ of the word~$\Qc$, the root configuration~$\Roots{I}$ is the linear basis of simple roots, and the coefficients~$\vcoeff[i][j][\sqc]$ are the coefficients of~$\posToRootsc(j)$, which is an almost positive root. The non-negativity for an arbitrary position~$i$ thus follows from Corollary~\ref{coro:compatibility-coef} since the orbit of the initial prefix~$\sqc$ under the rotation~$\tausqc$ cover all positions in~$\Qc$. 

For the second part of the lemma observe that if~$i$ and~$j$ belong to some $\sqc$-cluster~$I$ then the coefficient~$\vcoeff[i][j][\sqc]$ computed in terms of this cluster is clearly equal to zero.  
Moreover, by Proposition~\ref{coro:compatibility-coef} the coefficient~$\vcoeff[i][j][\sqc]=0$ if and only if~$\vcoeff[\tausqc i][\tausqc j][\sqc]=0$, and by Corollary~\ref{coro:rotation} $i$ and~$j$ belong to some $\sqc$-cluster if and only if~$\tausqc i$ and~$\tausqc j$ belong to some $\sqc$-cluster. Therefore, it is enough to prove the result in the case when~$i\in[n]$ belongs to the first~$n$ positions in~$\Qc$. The result in this case can be deduced from~\cite[Theorem~5.1]{CeballosLabbeStump}. This theorem states that if $i\in [n]$, then $i$ and~$j$ belong to some $\sqc$-cluster if and only if~$\posToRootsc(j)\in (\Phi_{\langle c_i \rangle})_{\geq -1}$ is an almost positive root of the parabolic root system that does not contain the root~$\alpha_{c_i}$. Since for~$i\in[n]$ the coefficient~$\vcoeff[i][j][\sqc]$ is the coefficient of the root~$\alpha_{c_i}$ in the almost positive root~$\posToRootsc(j)$, the result immediately follows.
\end{proof}

%%%%%%%%%%%%%%%%%%%%%%%%%%%%%%%%%%%%%%

\section{Geometric interpretations in types~$A$, $B$, $C$ and~$D$}
\label{sec:geom-interpretations}

In this section, we present geometric interpretations for denominator vectors in the classical types $A$, $B$, $C$, and $D$, with respect to any initial cluster seed, acyclic or not. 
More details on the geometric models for cluster algebras of classical types can be found in~\cite[Section~3.5]{FominZelevinsky-YSystems} and~\cite[Section~12]{FominZelevinsky-ClusterAlgebrasII}. 
In type $D$ we prefer to use an interpretation based on pseudotriangulations, which slightly differs from that of S.~Fomin and A.~Zelevinsky, and which we will describe in detail in a forthcoming paper. 

\begin{figure}[h]
	\centerline{\includegraphics[scale=.77]{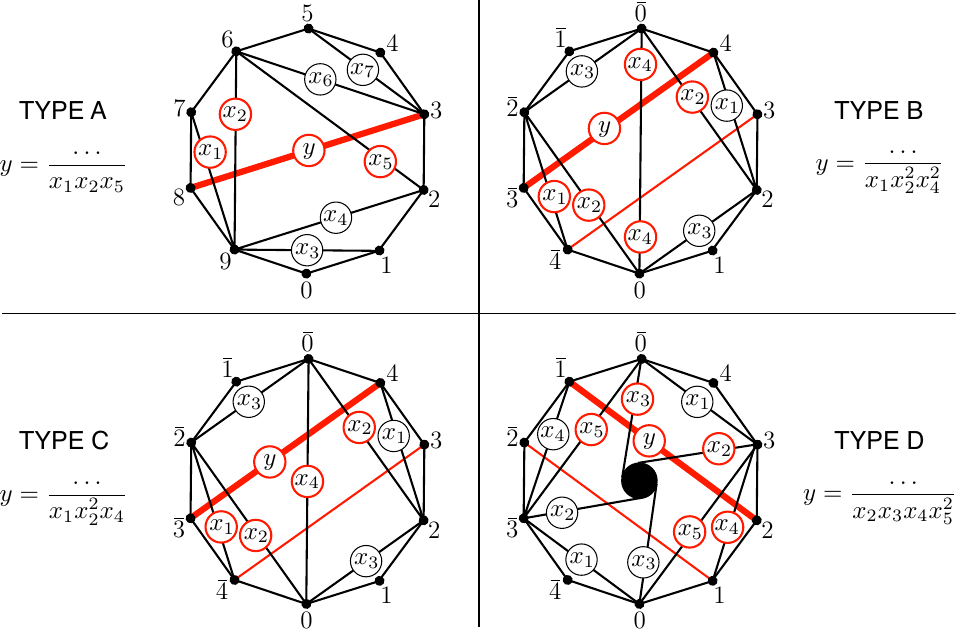}}
	\caption{Illustration of the geometric interpretations of denominator vectors in cluster algebras of classical types~$A$, $B$, $C$ and~$D$.}
	\label{fig:classicalTypes}
\end{figure}

As illustrated in \fref{fig:classicalTypes}, we can associate to each classical type a geometric configuration. Under this association, cluster variables correspond to \defn{diagonals} (or centrally symmetric pairs of diagonals) in the geometric picture, while clusters correspond to \defn{geometric clusters}: triangulations in type~$A$, centrally symmetric triangulations in types~$B$ and~$C$, and centrally symmetric pseudotriangulations in type~$D$ (\ie maximal crossing-free sets of centrally symmetric pairs of \textit{chords}). Moreover, cluster mutations correspond to \defn{geometric flips}, and the exchange relations on cluster variables can also be expressed geometrically.
Compatibility degrees and denominator vectors can be interpreted geometrically as follows.
Denote by~$\diagToVar$ the bijection from (c.s.~pairs of) diagonals to cluster variables. 
Given any two (c.s.~pairs of) diagonals~$\theta, \delta$, the compatibility degree~$d(\diagToVar(\theta),\diagToVar(\delta))$ between the corresponding cluster variables is given by the \defn{crossing number} $\cross[\theta][\delta]$ of the diagonals~$\theta$ and~$\delta$. By definition $\cross[\delta][\delta]=-1$, and if~$\theta \ne \delta$ then

\begin{itemize}
\item In type~$A$, $\cross[\theta][\delta]$ is equal to~$1$ if the diagonals~$\theta \ne \delta$ cross, and to~$0$ otherwise.
\item In type~$B$, we represent long diagonals (\ie diameters of the polygon) by doubled long diagonals. If~$\delta$ is not a long diagonal then~$\cross[\theta][\delta]$ is the number of times that a representative diagonal of the pair~$\delta$ crosses the pair~$\theta$. If~$\delta$ is a long diagonal then~$\cross[\theta][\delta]$ is~$1$ if~$\theta$ and~$\delta$ cross, and~$0$ otherwise. 
\item In type~$C$, the long diagonals remain as single long diagonals. The crossing number~$\cross[\theta][\delta]$ is the number of times that a representative diagonal of the pair~$\delta$ crosses~$\theta$.
\item In type~$D$, $\cross[\theta][\delta]$ is equal to the number of times that a representative diagonal of the pair~$\delta$ crosses the chords of~$\theta$.
\end{itemize}

Given any (centrally symmetric) seed triangulation~$T \eqdef \{ \theta_1, \dots ,\theta_n\}$ and any diagonal~$\delta$ (or c.s. pair of diagonals~$\delta$), the $d$-vector of the cluster variable~$\diagToVar(\delta)$ with respect to the initial cluster seed~$\diagToVar(T)$ is the \defn{crossing vector} 
\[
\dvector(T,\delta) \eqdef (\cross[\theta_1][\delta], \dots, \cross[\theta_n][\delta])
\] 
of~$\delta$ with respect to~$T$.

%%%%%%%%%%%%%%%%%%%%%%%%%%%%%%%%%%%%%%

\section*{acknowledgements}
\label{sec:ack}

The first author would like to thank Sergey Fomin, Bernhard Keller, Andrei Zelevinsky, and specially Salvatore Stella for fruitful discussions at the MSRI workshop \emph{Cluster Algebras in Combinatorics, Algebra and Geometry} in 2012. He also thanks the MSRI for providing a wonderful environment during the workshop.
The main ideas used in this paper originated while we were trying to find general families of polytopal realizations of generalized associahedra given any initial cluster seed. We thank Paco Santos and Christian Stump for many discussions on these~topics.     

\bibliographystyle{alpha}
\bibliography{CeballosPilaud_dvectors}
\label{sec:biblio}

%%%%%%%%%%%%%%%%%%%%%%%%%%%%%%%%%%%%%%

\end{document}